\documentclass[12pt]{amsart}
\usepackage{amssymb}
\newtheorem{theorem}{Theorem}

\newtheorem{corollary}[theorem]{Corollary}

\newtheorem{question}[theorem]{Question}
\newtheorem{observation}[theorem]{Observation}

\newtheorem{definition}[theorem]{Definition}

\newcommand{\QED}{\end{proof}}

\def\BF#1.{{\bf #1.}}
\newcommand{\url}[1]{{\tt #1}}
\newcommand{\cal}{\mathcal}

\renewcommand{\P}{{\mathbb P}}
\newcommand{\Q}{{\mathbb Q}}

\newcommand{\Z}{{\mathbb Z}}
\newcommand{\R}{{\mathbb R}}

\newcommand{\calM}{{\mathcal M}}

\newcommand{\Ptail}{{\P_{\!\scriptscriptstyle\rm tail}}}

\newcommand{\one}{\mathop{1\hskip-3pt {\rm l}}}
\newcommand{\of}{\subseteq}

\newcommand{\set}[1]{\{\,{#1}\,\}}

\newcommand{\elesub}{\prec}

\newcommand{\ran}{\mathop{\rm ran}}

\newcommand{\Add}{\mathop{\rm Add}}

\newcommand{\Coll}{\mathop{\rm Coll}}

\newcommand{\plus}{{+}}
\newcommand{\plusplus}{{{+}{+}}}

\newcommand{\satisfies}{\models}
\newcommand{\forces}{\Vdash}

\newcommand{\concat}{\mathbin{{}^\smallfrown}}

\newcommand{\Union}{\bigcup}

\newcommand{\intersect}{\cap}

\newcommand{\smalllt}{\mathrel{\mathchoice{\raise2pt\hbox{$\scriptstyle<$}}{\raise1pt\hbox{$\scriptstyle<$}}{\raise0pt\hbox{$\scriptscriptstyle<$}}{\scriptscriptstyle<}}}
\newcommand{\smallleq}{\mathrel{\mathchoice{\raise2pt\hbox{$\scriptstyle\leq$}}{\raise1pt\hbox{$\scriptstyle\leq$}}{\raise1pt\hbox{$\scriptscriptstyle\leq$}}{\scriptscriptstyle\leq}}}
\newcommand{\lt}{\smalllt}

\newcommand{\ltkappa}{{{\smalllt}\kappa}}

\newcommand{\ltgamma}{{{\smalllt}\gamma}}

\newcommand{\boolval}[1]{\mathopen{\lbrack\!\lbrack}\,#1\,\mathclose{\rbrack\!\rbrack}}
\def\[#1]{\boolval{#1}}

\newcommand{\UnderTilde}[1]{{\setbox1=\hbox{$#1$}\baselineskip=0pt\vtop{\hbox{$#1$}\hbox to\wd1{\hfil$\sim$\hfil}}}{}}
\newcommand{\Undertilde}[1]{{\setbox1=\hbox{$#1$}\baselineskip=0pt\vtop{\hbox{$#1$}\hbox to\wd1{\hfil$\scriptstyle\sim$\hfil}}}{}}
\newcommand{\undertilde}[1]{{\setbox1=\hbox{$#1$}\baselineskip=0pt\vtop{\hbox{$#1$}\hbox to\wd1{\hfil$\scriptscriptstyle\sim$\hfil}}}{}}
\newcommand{\UnderdTilde}[1]{{\setbox1=\hbox{$#1$}\baselineskip=0pt\vtop{\hbox{$#1$}\hbox to\wd1{\hfil$\approx$\hfil}}}{}}
\newcommand{\Underdtilde}[1]{{\setbox1=\hbox{$#1$}\baselineskip=0pt\vtop{\hbox{$#1$}\hbox to\wd1{\hfil\scriptsize$\approx$\hfil}}}{}}

\newcommand{\st}{\mid}
\renewcommand{\th}{{\hbox{\scriptsize th}}}

\def\<#1>{\langle#1\rangle}

\newcommand{\ORD}{\mathop{{\rm ORD}}}

\newcommand{\ZFC}{{\rm ZFC}}
\newcommand{\ZF}{{\rm ZF}}
\newcommand{\KM}{{\rm KM}}

\newcommand{\GBC}{{\rm GBC}}

\newcommand{\GCH}{{\rm GCH}}

\newcommand{\HOD}{{\rm HOD}}

\newcommand{\PA}{{\rm PA}}

\newcommand{\cell}[1]{\boxit{\hbox to 17pt{\strut\hfil$#1$\hfil}}}
\newcommand{\head}[2]{\lower2pt\vbox{\hbox{\strut\footnotesize\it\hskip3pt#2}\boxit{\cell#1}}}
\newcommand{\boxit}[1]{\setbox4=\hbox{\kern2pt#1\kern2pt}\hbox{\vrule\vbox{\hrule\kern2pt\box4\kern2pt\hrule}\vrule}}
\newcommand{\Col}[3]{\hbox{\vbox{\baselineskip=0pt\parskip=0pt\cell#1\cell#2\cell#3}}}
\newcommand{\tapenames}{\raise 5pt\vbox to .7in{\hbox to .8in{\it\hfill input: \strut}\vfill\hbox to
.8in{\it\hfill scratch: \strut}\vfill\hbox to .8in{\it\hfill output: \strut}}}
\newcommand{\Head}[4]{\lower2pt\vbox{\hbox to25pt{\strut\footnotesize\it\hfill#4\hfill}\boxit{\Col#1#2#3}}}
\newcommand{\Dots}{\raise 5pt\vbox to .7in{\hbox{\ $\cdots$\strut}\vfill\hbox{\ $\cdots$\strut}\vfill\hbox{\
$\cdots$\strut}}}
\newcommand{\df}{\it} 
\begin{document}
\author[Hamkins]{Joel David Hamkins}
\address{J. D. Hamkins, Mathematics,
The Graduate Center of The City University of New York, 365
Fifth Avenue, New York, NY 10016 \& Mathematics, The
College of Staten Island of CUNY, Staten Island, NY 10314}
\email{jhamkins@gc.cuny.edu, http://jdh.hamkins.org}
\thanks{The research of the first author has been
supported in part by grants from the National Science
Foundation, the CUNY Research Foundation and the Simons
Foundation. The research of the third author has been
supported in part by grants from the CUNY Research
Foundation.}
\author{David Linetsky}
\address{D. Linetsky, Mathematics,
The Graduate Center of The City University of New York, 365
Fifth Avenue, New York, NY 10016}
\email{dlinetsky@gmail.com}
\author{Jonas Reitz}
\address{J. Reitz, New York City College of Technology of The City University of New York, Mathematics, 300 Jay Street, Brooklyn, NY 11201}
\email{jonasreitz@gmail.com}
\subjclass[2000]{03E55}\keywords{set theory, forcing}
\begin{abstract}
A pointwise definable model is one in which every object is
definable without parameters. In a model of set theory,
this property strengthens $V=\HOD$, but is not first-order
expressible. Nevertheless, if \ZFC\ is consistent, then
there are continuum many pointwise definable models of
\ZFC. If there is a transitive model of \ZFC, then there
are continuum many pointwise definable transitive models of
\ZFC. What is more, every countable model of \ZFC\ has a
class forcing extension that is pointwise definable.
Indeed, for the main contribution of this article, every
countable model of G\"odel-Bernays set theory has a
pointwise definable extension, in which every set and class
is first-order definable without parameters.
\end{abstract}

\title{Pointwise Definable Models of Set Theory}
\maketitle

\section{Introduction by way of a curious logic conundrum}

One occasionally hears the argument---let us call it the
math-tea argument, for perhaps it is heard at a good math
tea---that there must be real numbers that we cannot
describe or define, because there are are only countably
many definitions, but uncountably many reals. Does it
withstand scrutiny?\footnote{See
\cite{MO44102:AnalysisInFactAnalysisOfDefinableNumbers} for
an instance of the argument at MathOverflow, which surely
serves a brisk cup of math tea online. We leave aside the
remark of Horatio, eight-year-old son of the first author,
who announced, ``Sure, papa, I can describe any number. Let
me show you: tell me any number, and I'll tell you a
description of it!''}

We can be precise, of course, and define that an element
$a$ of a structure $\mathcal M$ is {\df definable} (without
parameters) when there is a formula $\varphi(x)$ in the
language of $\mathcal M$ such that ${\mathcal M}\satisfies
\varphi[x]$ only at $x=a$. In the continuum of the ordered
real line $\<\R,{<}>$, for example, there are no definable
elements, since the automorphism group acts transitively by
translation and all points look alike. With the ordered
field structure $\<\R,+,\cdot,0,1,<>$, however, every
algebraic number becomes definable, but only these, by
Tarski's theorem on real closed fields. As one adds
additional structure
$\<\R,+,\cdot,0,1,<,\Z,\sin(x),e^x,\cdots>$, the collection
of definable reals grows larger. Eventually, we will in our
definitions be attracted to the possibilities of using
higher order mathematical objects and constructions, such
as function classes, spaces or measures, and this amounts
to defining objects in increasingly large fragments
$V_\alpha$ of the set-theoretic universe. Most all of the
classical mathematical structure is itself definable in the
set-theoretic structure $\<V_{\omega+\omega},{\in}>$, a
model of the Zermelo axioms, and so the definable reals of
this structure includes almost every real ever defined
classically. The structures arising with larger ordinals,
however, allow us to define even more reals.

It would be a kind of cheating, for the problem at hand, to
define a real $r$ or other object by using a language or
structure that was itself otherwise undefinable or
uncanonical, such as by using a constant symbol with value
$r$ or a unary predicate satisfied only at $r$; this would
be like defining $\pi$ as ``the value of the constant
symbol $\pi$,'' which is surely unsatisfying. Similarly, we
do not want to offer a definition of $r$ in some enormous
$\<V_\alpha,{\in}>$ when $\alpha$ is not itself somehow
definable, since in effect this is using $\alpha$ as a kind
of parameter. Rather, the spirit of the problem seems to be
to define reals using only structure that is itself
definable in the set-theoretic background. This amounts, of
course, simply to defining the real directly with respect
to the set-theoretic background $\<V,{\in}>$ in the first
place.

In any fixed structure $\cal M$ in a countable language,
including the higher-order set-theoretic structures
$\<V_\alpha,{\in}>$, the math-tea argument seems fine:
since there are only countably many definitions to use, but
uncountably many reals, there will indeed be many reals
that are not definable there.

But when we make the move as we have discussed to defining
reals or other objects with respect to the set-theoretic
background $\<V,{\in}>$, a subtle meta-mathematical
obstacle arises for the math-tea argument. Specifically, in
order to count the definable objects, the argument presumes
that we have a way of associating to each definable object
a definition of it, a definability map $r\mapsto\psi$,
where $r$ is the unique object satisfying $\psi(r)$. This
is not a problem when we are working with definitions over
a set structure, since the satisfaction relation over a set
structure is definable from that structure, but it becomes
problematic with definitions over $V$, for reasons
connected with Tarski's theorem on the non-definability of
truth. Basically, there is no uniform first-order way to
express the concept ``$x$ is defined by formula $\psi$''
within set theory. This obstacle for the math-tea argument
suggests a possibility that perhaps one could live in a
model of \ZFC\ in which every real is definable without
parameters.

\begin{question}\label{Question.AllRealsDefinable?}
Is it consistent with the axioms of set theory that every
real is definable in the language of set theory without
parameters?
\end{question}

The answer is Yes. Indeed, much more is true: if the \ZFC\
axioms of set theory are consistent, then there are models
of \ZFC\ in which every object, including every real
number, every function on the reals, every set of reals,
every topological space, every ordinal and so on, is
uniquely definable without parameters. Inside such a
universe, the math-tea argument comes ultimately to a false
conclusion.

\begin{definition}\rm
A first-order structure $\mathcal M$ is {\df pointwise
definable} if every element of $\mathcal M$ is definable in
$\mathcal M$ without parameters.
\end{definition}

In a pointwise definable model, every object can be
specified as the unique object with some first-order
property. In such models, all objects are discernible;
every object satisfies a unique principal complete
$1$-type. Notice that the property of being pointwise
definable is not first-order expressible, since it is not
preserved by elementary extensions. Clearly, pointwise
definable models in a countable language must be countable,
since there are only countably many definitions.

\begin{theorem}\ \label{MainTheorem}
 \begin{enumerate}
  \item If\/ \ZFC\ is consistent, then there are
      continuum many non-isomorphic pointwise definable
      models of \ZFC.
  \item If there is a transitive model of \ZFC, then
      there are continuum many transitive
      pointwise-definable models of \ZFC.
  \item Every countable model of \ZFC\ has a class
      forcing extension that is pointwise definable.
  \item Every countable model of \GBC\ has a pointwise
      definable extension, in which every set and class
      is first-order definable without parameters.
 \end{enumerate}
\end{theorem}


The first two claims have fairly soft proofs and might be
considered to be a part of the mathematical folklore.
Statement (1) could be credited to Myhill
\cite{Myhill1952:TheHypothesisThatAllClassesAreNameable}.
The latter two claims are more substantial forcing
arguments. Statement (3) is mentioned independently in
\cite{Enayat2005:ModelsOfSetTheoryWithDefinableOrdinals}
and, in the case of countable transitive models, in
\cite{David1982:SomeApplicationOfJensensCodingTheorem}. Our
main contribution is statement (4), which implies the
earlier statements. The rest of this article is devoted to
proving these facts and several other related results we
find interesting. We begin in section \ref{Section.PDM}
with some elementary observations about pointwise definable
models of set theory and prove the easier initial results.
In sections \ref{Section.PDForcingExtensions} and
\ref{Section.NGBC}, we move into the forcing arguments,
relying on work of Simpson in the case of \ZFC\ and a
result of S. Friedman, building on a result of Kossak and
Schmerl from \PA, for the fully general \GBC\ case.

\section{Pointwise Definable Models of
\ZFC}\label{Section.PDM}

The first task is to establish the basic fact that
pointwise definable models of set theory do indeed exist.

\begin{theorem}\label{Theorem.Con(ZFC)->PDM}
 If there is a model of \ZFC, then there is a pointwise
 definable model of \ZFC. Indeed, there are continuum many
 non-isomorphic such models.
\end{theorem}

\begin{proof}
If there is a model of \ZFC, then there is a model
$M\satisfies\ZFC+V=\HOD$. Such a model has a parameter-free
definable well-ordering of the universe, and therefore it
has parameter-free definable Skolem functions, which simply
select the least witness with respect to the definable
well-ordering. If $M_0$ is the collection of all definable
elements of $M$, then it is closed under these definable
Skolem functions and thus $M_0\elesub M$. From this, it
follows that definitions work the same in $M_0$ as in $M$,
and so every object of $M_0$ is definable in $M_0$. Thus,
$M_0$ is a pointwise definable model of \ZFC. By the
G\"odel-Rosser theorem, there are continuum many consistent
completions of $\ZFC+V=\HOD$, and we have established that
each such theory has a pointwise definable model.
\end{proof}

Myhill
\cite{Myhill1952:TheHypothesisThatAllClassesAreNameable}
was evidently the first to observe (in contemporary
language) that if \ZFC\ is consistent, then there is a
pointwise definable model of $\GBC+V=L$, by essentially the
argument we have given, using the definable Skolem
functions of $L$. Indeed, Myhill was fully aware of the
implications for the math-tea argument, for he concludes
his otherwise terse article with:
\begin{quote}
 One often hears it said that since there are indenumerably many
 sets and only denumerably many names, therefore there must
 be nameless sets. The above shows this argument to be
 fallacious.
\end{quote}

The fact is that every pointwise definable model of \ZFC\
arises in exactly the manner of the proof of theorem
\ref{Theorem.Con(ZFC)->PDM}. A {\df prime} model is one
that embeds elementarily into every model of its theory.

\begin{observation}\label{Observation.PDMiffDefinableElementsHOD}
The following are equivalent:
 \begin{enumerate}
  \item $M$ is a pointwise definable model of \ZFC.
  \item $M$ consists of the definable elements, without
      parameters, of a model of $\ZFC+V=\HOD$.
  \item $M$ is a prime model of $\ZFC+V=\HOD$.
 \end{enumerate}
\end{observation}

\begin{proof}
The first statement implies the second since every
pointwise-definable model satisfies $V=\HOD$, and the
converse implication is the proof of theorem
\ref{Theorem.Con(ZFC)->PDM}. These models are exactly the
prime models of $\ZFC+V=\HOD$, since the theory admits
definable Skolem functions, and so every model of
$\ZFC+V=\HOD$ elementarily embeds the definable hull of the
empty set in that model.
\end{proof}

More generally, if $M$ is any model of \ZF, then the
collection $M_0$ consisting of the elements of $M$ that are
definable without parameters in $M$ is an elementary
substructure of $\HOD^M$, because it is closed under the
canonical Skolem functions for $\HOD^M$, which are
definable in $M$. A similar fact is observed in
\cite[Theorem
2.11]{Enayat2005:ModelsOfSetTheoryWithDefinableOrdinals}.
The model $M_0$ is therefore prime relative to $M$, in the
sense that whenever $M\equiv N$, then $M_0$ embeds
elementarily in $\HOD^N$.

Note that if two pointwise definable models have the same
theory, then they are isomorphic, since the definitions of
the elements tell you the isomorphism. Next, we show that
there are also numerous well-founded pointwise definable
models of set theory, if there is a well-founded model of
set theory at all.

\begin{theorem}\label{Theorem.TransitiveModelZFC-->TransitivePDM}
If there is a transitive model of \ZFC, then there are
continuum many transitive pointwise-definable models of
\ZFC.
\end{theorem}

\begin{proof}
If there is a transitive model $M$ of \ZFC, then there is a
transitive model $N$ of $\ZFC+V=\HOD$, such as the model
$L^M$ or any of the models obtained by forcing $V=\HOD$
over $M$, and such models exist if $M$ is countable. By
observation \ref{Observation.PDMiffDefinableElementsHOD},
the collection of parameter-free definable elements of $N$
is a well-founded pointwise-definable elementary
substructure, whose Mostowski collapse is as desired.

To construct continuum many such models, suppose that $M$
is any countable transitive model of \ZFC. Since $M$ is
countable, there are a perfect set of $M$-generic Cohen
reals $c$, with which we may form the forcing extension
$M[c]$. By any of the usual methods, such as coding sets
into the \GCH\ pattern (which we will shall explain in
greater detail in section
\ref{Section.PDForcingExtensions}), we may form a further
forcing extension $M[c][G]$ satisfying $\ZFC+V=\HOD$. The
$M[c]$-generic filter $G$ is built in a diagonal
construction meeting the countably many dense classes of
$M[c]$. By ensuring that $c$ is the first set to be coded
into the \GCH\ pattern---we could arrange that the \GCH\
holds in $M[G]$ at $\aleph_n$ exactly when the $n^\th$
digit of $c$ is $1$---we may assume that $c$ is definable
without parameters in $M[c][G]$. Thus, if $M_0$ is the
collection of definable elements of $M[c][G]$, then it is
pointwise definable and contains $c$. Since there are
continuum many such reals $c$, we have thereby produced
continuum many transitive pointwise definable models of
\ZFC.
\end{proof}

A similar observation establishes the following folklore
result. The {\df minimal transitive model} of \ZFC\ is
$L_\alpha$ for the smallest ordinal $\alpha$ for which
$L_\alpha\satisfies\ZFC$, if any such ordinal exists. If
$M$ is any transitive model of \ZFC, then
$L^M\satisfies\ZFC$ and $L^M=L_\alpha$ for $\alpha=\ORD^M$,
and so the minimal model exists and is included within $M$,
thereby justifying its name.

\begin{theorem}\label{Theorem.MinimalModel}
The minimal transitive model of \ZFC\ is pointwise
definable. If $L_\beta$ is pointwise definable, then the
next $\hat\beta>\beta$ with $L_{\hat\beta}\satisfies\ZFC$,
if it exists, is also pointwise definable.
\end{theorem}

\begin{proof}
If $L_\alpha$ is the minimal transitive model of \ZFC, then
by condensation the definable hull of $\emptyset$ in
$L_\alpha$ collapses to $L_\alpha$, and so every element of
$L_\alpha$ is definable in $L_\alpha$. If $L_\beta$ is a
pointwise definable model of \ZFC\ and $\hat\beta$ is the
smallest ordinal above $\beta$ for which
$L_{\hat\beta}\satisfies\ZFC$, then $\beta$ is definable in
$L_{\hat\beta}$ as the largest ordinal such that
$L_\beta\satisfies\ZFC$. It follows that every element of
$L_\beta$ is definable in $L_{\hat\beta}$, by using a
definition relativized to $L_\beta$, which is definable in
$L_{\hat\beta}$. If $D$ is the set of definable elements of
$L_{\hat\beta}$, then the Mostowski collapse of $D$ will
have height larger than $\beta$, and be a model of
$\ZFC+V=L$, so by the minimality of $\hat\beta$ it will be
all of $L_{\hat\beta}$. Thus, $L_{\hat\beta}$ is pointwise
definable, as desired.
\end{proof}

The phenomenon of theorem \ref{Theorem.MinimalModel}
extends to describable limits of such $L_\alpha$. For
example, the first $L_\beta\satisfies\ZFC$ for which
$\beta$ is a limit of $\alpha$ for which
$L_\alpha\satisfies\ZFC$ is also pointwise definable,
because the definable hull of $L_\beta$ will collapse to
such a limit and therefore collapse to $L_\beta$ itself.
And similarly for many other limits. But if $L_\alpha$ is
an elementary substructure of $L_\beta$, for
$\alpha\ne\beta$, then of course $L_\beta$ cannot be
pointwise definable, since the definable elements must lie
in the range of the embedding. Similarly, if
$\omega_1^L\leq\alpha$, then $L_\alpha$ cannot be pointwise
definable, irrespective of whether it satisfies \ZFC\ or
not, since this property is absolute to $L$ and such
$L_\alpha$ are uncountable in $L$.

\begin{theorem}\label{Theorem.ArbLargeLalphaPointwiseDefinable}
There are arbitrarily large $\xi<\omega_1^L$ for which
$\<L_\xi,\in>$ is pointwise definable. If there are
arbitrarily large $\alpha<\omega_1^L$ for which
$L_\alpha\satisfies\ZFC$, then there are arbitrarily large
such $\alpha$ for which $L_\alpha\satisfies\ZFC$ and is
pointwise definable.
\end{theorem}

\begin{proof} The second claim is essentially \cite[thm
3.7]{Enayat2002:CountingModelsOfSetTheory}. For the first
claim, observe that every real of $L$ is definable without
parameters in some countable $L_\xi$, because the $L$-least
real $z$ not definable in any countable $L_\xi$, if any
should exist, is thereby definable in $L_{\omega_1^L}$,
which will condense to a definition of $z$ in some
countable $L_\xi$, giving a contradiction. Indeed, by
taking the definable hull of the empty set, we find that
$z$ is definable in a pointwise definable $L_\xi$. It
follows that such $\xi$ must be unbounded in $\omega_1^L$,
establishing the first claim. For the second claim, suppose
that there are unboundedly many $\alpha<\omega_1^L$ for
which $L_\alpha\satisfies\ZFC$. Observe that every
countable ordinal $\xi<\omega_1^L$ is definable without
parameters in $L_\alpha\satisfies\ZFC$, where $\alpha$ is
the $(\xi+1)^\th$ ordinal for which
$L_\alpha\satisfies\ZFC$, because this $L_\alpha$ can see
exactly $\xi$ many smaller $\beta$ for which
$L_\beta\satisfies\ZFC$, an argument that recalls the
definition of Laver functions for uplifting cardinals in
\cite{HamkinsJohnstone:ResurrectionAxioms}. Combining the
previous facts, every real $z$ of $L$ is definable in some
$L_\xi$, which is itself definable in some
$L_\alpha\satisfies\ZFC$, and so $z$ is in the definable
hull of $L_\alpha$, which condenses to a pointwise
definable $L_{\alpha_0}\satisfies\ZFC$ containing $z$.
Since $z$ was arbitrary in $L$, these $\alpha_0$ must be
unbounded in $\omega_1^L$, as desired.
\end{proof}

A model $M$ of \ZF\ is {\df $\omega$-standard} if the
natural numbers of $M$ have the standard order-type
$\omega$. More generally, $M$ is $\zeta$-standard for an
ordinal $\zeta$ if the ordinals of $M$ have a $\zeta^\th$
member, so that the well-founded initial segment of the
ordinals of $M$ has order type exceeding $\zeta$.

\begin{theorem}\label{Theorem.EveryRealInPDM}
If there are arbitrarily large $\alpha<\omega_1^L$ for
which $L_\alpha\satisfies\ZFC$, then every real in $V$ is
an element of a pointwise definable $\omega$-standard model
of $\ZFC+V=L$, whose well-founded initial segment can be as
large in the countable ordinals of $V$ as desired.
\end{theorem}

\begin{proof}
The argument of theorem
\ref{Theorem.ArbLargeLalphaPointwiseDefinable} shows under
the hypothesis that the conclusion is true in $L$. We
complete the proof of the theorem by observing that the
statement that every real of $V$ is in a pointwise
definable $\omega$-standard model of $\ZFC+V=L$ has
complexity $\Pi^1_2$, since it has the form ``{\it for
every real, there is a countable structure...},'' where the
omitted portion is arithmetic, since it involves
quantification only over the elements of the structure and
the digits of the real. Thus, by Shoenfield absoluteness
this statement is absolute from $L$ to $V$ and hence is
true in $V$, as desired, even if $V$ has many more reals
and larger countable ordinals than $L$. Note that if an
$\omega$-standard model $M$ has a real $z$ coding a
countable ordinal $\zeta$, then $M$ will decode this real
correctly, and so $M$ will be well-founded beyond $\zeta$,
and the theorem is proved.
\end{proof}

\begin{corollary}\label{Corollary.EveryMextendsToPointwiseDefinableM+}
If there are arbitrarily large $\alpha<\omega_1^L$ with
$L_\alpha\satisfies\ZFC$, then every countable transitive
set $M$ is a countable transitive set inside a structure
$M^+$ that is a pointwise definable model of $\ZFC+V=L$,
and $M^+$ is well-founded beyond the rank of $M$.
\end{corollary}

\begin{proof}
Code $M$ by a real $z$, and place $z$ inside a pointwise
definable $\omega$-standard model $M^+\satisfies\ZFC+V=L$.
Since $M^+$ performs the Mostowski collapse of the
structure coded by $z$ correctly, it follows that $M$ is a
transitive set in $M^+$. (In this sense, $M$ end-extends to
$M^+$.) In particular, $M^+$ is well-founded beyond the
rank of $M$.
\end{proof}

Let us notice a few things about these arguments. First,
although in $L$ we produced fully well-founded pointwise
definable models, if in theorem
\ref{Theorem.EveryRealInPDM} we were to have made that part
of the statement, then the complexity would have risen to
$\Pi^1_3$ and we would have lost the absoluteness that
allowed us to bring the statement from $L$ to $V$. And
clearly, we cannot expect to place non-constructible reals
inside well-founded models of $V=L$. Second, the hypothesis
that there are arbitrarily large $\alpha<\omega_1^L$ with
$L_\alpha\satisfies\ZFC$ follows from the simpler (but
strictly stronger) hypothesis that there is a single
uncountable transitive model of \ZFC, because one may
consider increasingly large countable elementary
substructures of the $L$ of such a model. We may omit these
hypotheses completely in the theorems above, if it is
desired only to have $\omega$-standard pointwise definable
models of $\ZFC^*+V=L$, where $\ZFC^*$ is some finite
fragment of \ZFC, rather than full \ZFC, since by
reflection and condensation there are arbitrarily large
countable $L_\alpha$ satisfying any such finite fragment
$\ZFC^*$.

The results become rather curious when there are reals in
$V$ that are not in $L$. Suppose there are sufficient
$L_\alpha\satisfies\ZFC$ and we force to collapse
$\omega_1$, with the resulting forcing extension $V[g]$,
where $g$ is a real coding $\omega_1^V$. By the argument
above, there is in $V[g]$ a countable $\omega$-standard
model $M\satisfies\ZFC+V=L$ in which $g$ exists and is
actually definable. Since $M$ has the standard $\omega$, it
interprets $g$ correctly and so the well-founded part of
$M$ exceeds $\omega_1^V$. But since $M\satisfies V=L$, it
thinks that $g$ is constructible as well as definable,
constructed at some (nonstandard) ordinal stage
$\tilde\alpha>\omega_1^V\geq\omega_1^L$!

Another interesting example arises when $0^\sharp$ exists.
This assumption implies the hypothesis of the theorem, and
so by the theorem there is a pointwise definable model $M$
of $\ZFC+V=L$, well-founded as high in the countable
ordinals as desired, in which the true $0^\sharp$ is a
member, thought by $M$ to be constructible at some stage
$\tilde\alpha$, necessarily in the ill-founded part of $M$.
Thus, the true $0^\sharp$ exists unrecognized but definable
in a model of $\ZFC+V=L$ that is well-founded far beyond
the true $\omega_1^L$! For example, the model $M$ can be
well-founded beyond a rich class of (countable) Silver
indiscernibles (or even beyond $\omega_1^V$ if one forces
to collapse as in the previous paragraph) and have the same
theory as the true $L^V$. The true ordinal indiscernibles
of $V$ become discernible in $M$, however, since it is
pointwise definable.

Corollary
\ref{Corollary.EveryMextendsToPointwiseDefinableM+} may
seem paradoxical when applied to a countable transitive
model $M$ having enormous large cardinals incompatible with
$L$. For example, perhaps $M$ has supercompact cardinals
inside it or has measures that are iterable in $V$, meaning
all iterates are well-founded; they cannot be iterable in
$M^+$, since this is a model of $V=L$. In any case,
corollary
\ref{Corollary.EveryMextendsToPointwiseDefinableM+} already
implies that every countable model of \ZFC\ or \GBC\ admits
an end-extension to a pointwise definable model of set
theory. In theorems \ref{Theorem.PDForcingExtension} and
\ref{Theorem.NGBC}, however, we shall find pointwise
definable extensions with the same ordinals.

Let us now turn to the question of the extent to which
definability is first-order expressible, by making a number
of observations that illustrate the range of possibility.
We have already noted that the property of a model being
pointwise definable is not first-order expressible, since
it is not preserved by nontrivial elementary extensions.
Since pointwise definability is a strong generalization of
the axiom $V=\HOD$, it is tempting to introduce such
notation as $V=D$ or $V=HD$ to express that a model is
pointwise definable, thereby maintaining a parallel to the
classical $V=\HOD$ notation while emphasizing that the
definitions need no parameters. We hesitate to adopt this
notation, however, because we fear it would incorrectly
suggest that the concept is first-order expressible, which
isn't the case.

(i) {\it There is no uniform definition of the class of
definable elements.} Specifically, there is no formula
$\mathop{\rm df}(x)$ in the language of set theory that is
satisfied in any model $M\satisfies\ZFC$ exactly by the
definable elements. The reason is that if $M_0$ is
pointwise definable and $M_0\elesub M$ is a nontrivial
elementary extension, then the definable elements of $M_0$
and $M$ are precisely the elements of $M_0$, and so $M_0$
should satisfy $\forall x\,\mathop{\rm df}(x)$ but $M$
would satisfy $\exists x\,\neg \mathop{\rm df}(x)$,
contrary to $M_0\elesub M$.

(ii) {\it In some models of set theory, the class of
definable elements is a definable class.} Although there is
no uniform definition of the class of definable elements,
it can sometimes happen that a model enjoys a certain
structure that allows it to see its collection of definable
elements as a definable class. For example, in a pointwise
definable model, the class of definable elements includes
every object and is therefore defined by the formula $x=x$.
See also (iv) and (v) below.

(iii) {\it In other models, the collection of definable
elements is not a class.} Consider any pointwise definable
model $M\satisfies\ZFC$, and let $N$ be an ultrapower of
     $M$ by a nonprincipal ultrafilter. The
     parameter-free definable elements of $N$ are
     exactly the elements in the range of the embedding.  If this
     collection were a class in $N$ then $N$ could reconstruct $M$ and
     realize itself as an ultrapower, which is impossible.

(iv) {\it In some models, the definable elements form a
definable class, but there is no class function
$r\mapsto\psi_r$ mapping definable elements to definitions
of them.} Suppose that $M$ is a pointwise definable model
of \ZFC. The definable elements of $M$ are all of $M$,
which is certainly a definable class in $M$. But $M$ cannot
have a function $r\mapsto\psi_r$ associating to each
element $r$ of $M$, or even to each real of $M$, a defining
formula $\psi_r$ of $r$, since such a map would reveal to
$M$ that it has only countably many objects.

(v) {\it In other models, the definable elements are a set
and there is a set definability map $r\mapsto\psi_r$.}
Suppose that $\kappa$ is an inaccessible cardinal (this
hypothesis can be reduced), and observe by a
Lowenheim-Skolem argument that there are numerous
$\gamma<\kappa$ with $V_\gamma\elesub
V_\kappa\satisfies\ZFC$. It follows that the definable
elements of $V_\kappa$ are all in $V_\gamma$ and satisfy
the same definitions there as in $V_\kappa$. Since
$V_\gamma$ is a set in $V_\kappa$, we may construct in
$V_\kappa$     the function $r\mapsto \psi_r$ that maps
every     definable element $r$ of $V_\gamma$ to the
smallest definition $\psi_r$ of it, and because
$V_\gamma\elesub V_\kappa$, this function has the same
property with respect to $V_\kappa$, as desired. The large
cardinal hypothesis can be reduced; it is sufficient to
have an $\omega$-model $M$ with some $M_0\in M$ having
$M_0\elesub M$.

(vi) {\it No model can have a {\it definable} definability
map $r\mapsto\psi_r$.} If such a map were definable, then
since there are only countably many definitions $\psi_r$,
we could easily diagonalize against it to produce a
definable real not in the domain of the map. In (v), the
map is definable from parameter $\gamma$.

The surviving content of the math-tea argument seems to be
the observation that in any model with access to a
definability map $r\mapsto\psi_r$, the definable reals do
not exhaust all the reals.

\section{Pointwise definable forcing
extensions}\label{Section.PDForcingExtensions}

Let us now move beyond the elementary methods and results
of section \ref{Section.PDM}, and prove that every
countable model $M$ of \ZFC\ has a carefully chosen class
forcing extension $M[G]$ that is pointwise definable, so
that all objects of $M$, as well as those in $M[G]$, become
definable in $M[G]$ without parameters. The forcing is
sufficiently adaptable so as to preserve all the usual
large cardinal axioms.

\begin{theorem}\label{Theorem.PDForcingExtension}
Every countable model of \ZFC\ has a pointwise definable
class forcing extension.
\end{theorem}

After proving this theorem (on a New York City subway
platform), we came later to learn of earlier work achieving
it. The introduction of
\cite{David1982:SomeApplicationOfJensensCodingTheorem}, for
example, claims the result in the special case of countable
transitive models as a corollary to its main theorem, but
the article unfortunately provides little further
explanation.\footnote{The main theorem of
\cite{David1982:SomeApplicationOfJensensCodingTheorem},
using Jensen coding, is that every transitive model of set
theory has a class forcing extension in which $V=L(r)$ for
a real $r$, such that $L_\alpha(r)\not\satisfies\ZF$ for
every ordinal $\alpha$. Since this situation is captured by
a first-order theory which can hold in uncountable models,
however, it cannot by itself imply pointwise definability;
David apparently had in mind an appeal to further details
of the proof. We note that in comparison to Jensen coding,
our forcing is mild: it is progressively closed and can be
made to preserve any of the usual large cardinals.} Also,
Ali Enayat in remark 2.8.1 of
\cite{Enayat2005:ModelsOfSetTheoryWithDefinableOrdinals}
briefly suggests a proof method very similar to our proof
of theorem \ref{Theorem.PDForcingExtension}, although he
does not state the result as a formal theorem. Most of
Enayat's excellent paper is concerned with the {\df Paris}
models, models of \ZF\ in which every ordinal is definable
without parameters, but in several instances he achieves
this by establishing pointwise definability. One of his
main results is \cite[theorem
2.19]{Enayat2005:ModelsOfSetTheoryWithDefinableOrdinals},
which asserts that if there is an uncountable transitive
model of \ZFC, then for every infinite cardinal $\kappa$,
there is a Paris model of \ZF\ having size $\kappa$. These
are very large models, but have only countably many
ordinals, since each ordinal is definable without
parameters.

Our proof of theorem \ref{Theorem.PDForcingExtension} is
based on a technique of Simpson
\cite{Simpson1974:ForcingAndModelsOfArithmetic}, an early
application of forcing to the study of models of
arithmetic. Namely, Simpson proved that every countable
model $M$ of \PA\ or \ZFC\ has an amenable class $U$
(i.e.~for any $x\in M$, $x\intersect U\in M$) with the
property that every element of the expanded structure
$\<M,U>$ is definable without parameters. Simpson uses
forcing to define a new generic proper class $U$ which
codes every element of the universe. The original structure
is expanded by adding a predicate for this new class,
producing a pointwise definable model in the larger
language. Our strategy will be simply to follow this
expansion by further forcing that codes $U$ and the new
generic filter into the first-order structure of the model,
thereby eliminating the need for the additional predicate
$U$ and producing a pointwise definable model in the
original pure language of set theory. In other words, we
will show that every countable model of \ZFC\ can be
extended to a pointwise definable model of \ZFC, a model in
which every set is definable without parameters.

We begin with a brief review of Simpson's
\cite{Simpson1974:ForcingAndModelsOfArithmetic} result.
Although his main result concerned models of PA, he
concludes his paper with a \ZFC\ version of the theorem,
and the proof given below is a straightforward adaptation
of his argument to the \ZFC\ context.

\medskip
\begin{theorem}[Simpson]\label{Theorem.Simpson}
Let $\<M,\in>$ be a countable model of \ZFC.  Then, there
is a class $U\subseteq M$ such that:
\begin{enumerate}
    \item[(i)] $\<M,{\in},U>$ satisfies \ZFC\ in the
        language with a predicate for $U$.
    \item[(ii)] Every element of $M$ is first-order
        definable in $\<M,{\in},U>$.
\end{enumerate}
\end{theorem}
\begin{proof}
Let us begin by enumerating the countable structure
$M=\set{a_n\st n<\omega}$. Using the axiom of choice in
$M$, it easy to see that every set $a_n\in M$ is coded in
$M$ by a binary sequence, say $\bar a_n\in 2^\alpha$, for
some ordinal $\alpha\in M$. Indeed, we can arrange things
in such a way so that all of the relevant coding is done
only on the even digits of $\bar a_n$, while all odd digits
are $0$, except that the sequence $\bar a_n$ ends with two
consecutive 1's immediately following a limit ordinal.
These restrictions on codes will be useful when we
concatenate many such sequences.

Consider the class forcing in $M$ to add a Cohen class of
ordinals, that is, the partial order $\Q=2^{<\ORD}$
consisting of all ordinal length binary sequences, ordered
by extension. Since $M$ is countable, we may enumerate the
dense subclasses $\<D_n\st n<\omega>$ of $\Q$ that are
definable in $M$ from parameters. Let us suppose $D_n$ is
defined by the formula $\varphi_n(x)$, using parameters
from $M$, and by padding if necessary, we may assume that
the parameters used in $\varphi_n$ are among $\{a_j\st
j<n\}$.

We now construct a certain $M$-generic filter for $\Q$ that
will meet all the dense sets $D_n$ while simultaneously
coding all the elements $a_n$ of $M$. Specifically, we
recursively define a descending sequence of conditions
$\<p_n\st n<\omega>$ in $\Q$, beginning with
$p_0=\emptyset$. At odd stages, let $p_{2n+1}$ be a
minimal-length extension of $p_{2n}$ such that $p_{2n+1}
\in D_n$. At even stages, let $p_{2n+2}={p_{2n+1}}\concat
\bar a_n$, the sequence obtained by concatenating the code
$\bar a_n$ to the end of $p_{2n+1}$. Finally, let
$U=\bigcup_{n\in\omega}p_n$. This is the union of the
filter consisting of the initial segments of $U$, which by
construction is $M$-generic for $\Q$, since we met each
dense class $D_n$. The forcing $\Q$ is $\kappa$-closed for
every $\kappa$ in $M$, and so it follows by standard
class-forcing arguments that we retain \ZFC\ in the
language with $U$ in the structure $\<M,{\in},U>$.

A straightforward inductive argument now shows that every
$a_n$ and $p_n$ is definable in this structure
$\<M,{\in},U>$. To see this, suppose that $p_{2n}$ and
$a_j$ for $j<n$ have already been defined. Then, $p_{2n+1}$
is simply defined as the least initial segment $q$ of $U$
extending $p_{2n}$ such that $M\models \varphi_n(q)$.
Notice that $\varphi_n$ uses only parameters from $a_j$ for
$j<n$, which have already been defined, so this definition
can be carried out without the use of any parameters. Now,
we define $p_{2n+2}$ as the least initial segment $q$ of
$U$, extending $p_{2n+1}$, that ends in two consecutive 1's
following a limit ordinal. Given $p_{2n+1}$ and $p_{2n+2}$,
it is a trivial matter to define $\bar a_n$, and $a_n$ is
easily defined as the unique set coded by $\bar a_n$. In
this fashion, every element of the universe $M$ can be
given a first-order parameter-free definition in the
structure $\<M,{\in},U>$.
\end{proof}

In the proof of theorem \ref{Theorem.Simpson}, we took care
in the selection of $U$ that the augmented structure
$\<M,{\in},U>$ would be pointwise definable. Let us briefly
argue that such extra care is indeed required in general,
for not every generic filter $U\of 2^{\lt\ORD}$ over a
countable structure $M$ need give rise to a pointwise
definable structure $\<M,{\in},U>$. One easy way to see
this is simply to let $\<N,{\in^*},U^*>$ be an internal
ultrapower of $\<M,{\in},U>$ by a nonprincipal ultrafilter
on $\omega^M$. Since the definable elements must be in the
range of this map, it follows that $\<N,{\in^*},U^*>$ is
not pointwise definable. But by elementarity, $U^*$ is
$N$-generic for the forcing $2^{\lt\ORD}$ in $N$.
Alternatively, one may construct transitive counterexamples
by starting with an uncountable $\<M,{\in}>\satisfies\ZFC$,
whose generic expansions $\<M,{\in},U>$, obtained by
forcing over $V$, remain uncountable and therefore not
pointwise definable, but further forcing over $V$ can make
them countable, while preserving the $M$-genericity of $U$.
So extra care is indeed required. In general, of course, if
one starts with a pointwise definable structure $M$, then
any expansion $\<M,{\in},U>$ of it will remain pointwise
definable, and indeed, for this one only needs that $M$ is
a Paris model (where every ordinal is definable without
parameters), since generically $U$ will list every set of
ordinals of $M$ and we will be able to define the beginning
and ending points in $U$ of a code for any desired set.

We now use Simpson's result to prove theorem
\ref{Theorem.PDForcingExtension}, where the pointwise
definable structure is not obtained in an {\it expansion}
of the original model $M$, by adding extra predicates to
the language, but rather in an extension of $M$, obtained
by forcing. Thus, we enlarge $M$ to a structure $M[G]$ in
which every object is definable without parameters in the
pure language of set theory.

\begin{proof}[Proof of theorem
\ref{Theorem.PDForcingExtension}] Let $\<M,{\in}>$ be a
countable model of \ZFC. By forcing over $M$ if necessary,
we may assume without loss of generality that
$\<M,{\in}>\models\GCH$. By theorem \ref{Theorem.Simpson},
there is an amenable class $U\subseteq M$ so that every
element $a\in M$ is first-order definable without
parameters in $\<M,{\in},U>$, which satisfies \ZFC\ in the
language with $U$. We now rid ourselves of the extra class
$U$ by forcing to code it into the structure of the
universe. Specifically, consider in $M$ the definable
sequence of uncountable regular cardinals
$\delta_\alpha=\aleph_{\omega\cdot\alpha+1}$, conveniently
separated by large gaps that will avoid interference, and
let $\P=\Pi_{\alpha\in
U}\Add(\delta_\alpha,\delta_\alpha^\plusplus)$ be Easton's
forcing, which forces failures of the \GCH\ precisely at
the cardinals $\delta_\alpha$ for $\alpha\in U$. If
$G\of\P$ is $M$-generic, then it is well known that
$M[G]\satisfies\ZFC$ and all cardinals and cofinalities are
preserved. It follows that the map
$\alpha\mapsto\delta_\alpha$ is definable in $M[G]$, and so
$U$ is definable in $M[G]$ by the equivalence $\alpha\in U$
if and only if $2^{\delta_\alpha}>\delta_\alpha^\plus$.

We shall now perform further forcing to ensure $V=\HOD$ in
a further extension $M[G][H]$, by coding into the \GCH\
pattern at the regular cardinals
$\gamma_\alpha=\aleph_{\omega\cdot\alpha+5}$, which sit
conveniently in the gaps of the $\delta_\alpha$ sequence,
again avoiding interference. Specifically, let $\R$ be the
Easton support class-length forcing iteration, which at
stage $\alpha$ generically decides either to force with
trivial forcing, which will preserve the \GCH\ at
$\gamma_\alpha$, or with
$\Add(\gamma_\alpha,\gamma_\alpha^\plusplus)$, which will
violate the $\GCH$ at $\gamma_\alpha$. That is, the stage
$\alpha$ forcing is the lottery sum
$\{\bar\one\}\oplus\Add(\gamma_\alpha,\gamma_\alpha^\plusplus)$,
which is the disjoint union of these partial orders joined
by a new common upper bound $\one$, so that any generic
filter picks exactly one factor and forces with it.
Since the forcing at stage $\alpha$ is
$\ltgamma_\alpha$-closed, we may nicely factor the
iteration $\R$ at arbitrarily large cardinals as set
forcing followed by highly closed tail forcing, and so in
the terminology of \cite{Reitz2006:Dissertation} it is
progressively closed and therefore preserves \ZFC. Note
also that $\R$ preserves all cardinals and cofinalities
over $M[G]$. To see that $\R$ forces $V=\HOD$, we observe
simply that every set of ordinals will eventually be coded
into the \GCH\ pattern on the coding points
$\gamma_\alpha$, because we can extend any condition to a
stronger condition that opts for the trivial or nontrivial
side of the forcing in the correct pattern so as to code
that set on an interval. Another way to say it is that the
necessary bookkeeping is performed generically by a density
argument. Specifically, suppose that $p\in\R$ and $\tau$ is
a $\R$-name such that $p\forces\tau\of\check\gamma$ for
some ordinal $\gamma$. Let $\beta$ be large enough so that
it is beyond the support of $p$, beyond $\gamma$ and also
large enough so that $\tau$ is a $\R_\beta$-name. Extend
$p$ to a stronger condition $q$, which for $\xi<\gamma$
opts for trivial forcing at stage $\beta+\xi$ with the same
Boolean value as $\boolval{\check\xi\in\tau}$ and for
nontrivial forcing at stage $\beta+\xi$ with the Boolean
value of $\boolval{\check\xi\notin\tau}$. That is,
$q(\beta+\xi)$ is an $\R_{\beta+\xi}$-name for the
condition that chooses one way or the other in the lottery
sum at that stage, depending on whether
$\xi\in\tau_{H_\beta}$, where $H_\beta$ is the generic
filter up to stage $\beta$. The condition $q$ therefore
forces that the set named by $\tau$, whatever it is, will
be coded into the \GCH\ pattern of the $\gamma_{\beta+\xi}$
for $\xi<\gamma$. Thus, if $H\of\R$ is $M$-generic, every
set of ordinals in $M[G][H]$ will be ordinal-definable, and
since every set can be easily coded by a set of ordinals,
it follows that $M[G][H]\satisfies V=\HOD$.

Let us now complete the argument. Since cardinals were
preserved, it follows that the map
$\alpha\mapsto\delta_\alpha$ is definable in $M[G][H]$.
Since the $\R$ forcing only affects the \GCH\ pattern at
the $\gamma_\alpha$, it does not upset the coding we did in
$M[G]$ at the $\delta_\alpha$, and so the class $U$ remains
definable in $M[G][H]$ as the class of $\alpha$ for which
the \GCH\ fails at $\delta_\alpha$. Note that the class $M$
is definable from $U$ in $M[G][H]$ as the class of sets
coded by a subsequence of $U$. Since $\<M,{\in},U>$ is
pointwise definable, it now follows that every element of
$M$ is definable in $M[G][H]$ without parameters. In
particular, every ordinal is definable in $M[G][H]$ without
parameters. Since $M[G][H]$ is a model of $V=\HOD$, this
implies that every element of $M[G][H]$ is definable
without parameters.
\end{proof}

One can easily combine the two steps of forcing in the
proof of theorem \ref{Theorem.PDForcingExtension} by
interleaving them into one forcing iteration, which codes
$U$ into the continuum pattern at $\delta_\alpha$ and holds
lotteries at $\gamma_\alpha$, leading in the same way to a
forcing extension satisfying $V=\HOD$ in which $U$ is
definable. We chose for clarity to separate the coding into
two steps.

The argument can be easily modified to use other coding
methods. For example, by using the $\Diamond^*$ coding
method as in
\cite{Brooke-Taylor2009:LargeCardinalsAndDefinableWellOrders}
one may also achieve \GCH\ in the final pointwise definable
model.

Note also that the forcing of theorem
\ref{Theorem.PDForcingExtension} is very mild from the
large cardinal perspective, for we first add a Cohen
generic class of ordinals, and then code it into the \GCH\
pattern while also forcing $V=\HOD$. These iterations are
extremely nice, progressively closed Easton support
iterations and can be made to have increasingly large gaps
where no forcing occurs. Thus, they can easily be made to
preserve any of the usual large cardinal notions. One may
therefore find pointwise definable forcing extensions while
preserving one's favorite large cardinals. For example, if
$M$ has a supercompact cardinal $\kappa$, one should first
make $\kappa$ Laver-indestructible, and then do all the
coding above $\kappa$, which will preserve the
supercompactness of $\kappa$ since it is
$\ltkappa$-directed closed.

\section{Extending the result from \ZFC\ to \GBC}
\label{Section.NGBC}

For the main contribution of this article, we should like
now to extend the result from \ZFC\ to G\"odel-Bernays
\GBC\ set theory, which is a natural context for pointwise
definability since it allows for a precise formal treatment
of classes in set theory and many of the pointwise
definability arguments have involved augmenting a model of
set theory with additional proper classes, such as those
arising with class forcing. We refer the reader to
\cite{Jech:SetTheory3rdEdition} or \cite[p.
225--86]{Mendelson1997:IntroductionToMathLogic} for an
overview of G\"odel-Bernays set theory. Here, we regard a
\GBC\ model of set theory as a triple $\calM=\<M,S,{\in}>$,
consisting of a model $\<M,{\in}>$ of \ZFC\ augmented with
a family $S\of P(M)$, whose members are the {\df classes}
of $\calM$, satisfying the \GBC\ axioms. The \GBC\ axioms
extend the \ZFC\ axioms to this two-sorted second-order
context by allowing formulas in the replacement and
separation axioms to make use of finitely many class
parameters (but not to quantify over classes, a
strengthening that constitutes Kelly-Morse set theory). In
particular, if $\vec X$ is any finite list of classes from
$S$, then the expansion $\<M,{\in},\vec X>$, obtained by
interpreting the classes in $\vec X$ as predicates over
$M$, satisfies $\ZFC(\vec X)$, the version of \ZFC\ in
which the classes of $\vec X$ may appear as atomic
predicates in the replacement and separation axioms
schemes, and furthermore, any class definable in
$\<M,{\in},\vec X>$ must be in $S$. In addition, \GBC\
includes a global version of the axiom of choice, meaning
that there is a single class choice function $F\in S$ such
that $F(x)\in x$ for every nonempty set $x$.

Although the \GBC\ axioms admit an elegant and economical
formulation purely in terms of classes, as every model of
\GBC\ is determined entirely by its collection of classes,
we prefer the two-sorted set-class formulation here so as
to emphasize the connection with models of \ZFC. In
particular, any model $\<M,{\in}>$ of \ZFC\ can be extended
to a model $\<M,S,{\in}>$ of \GBC\ by first finding an
$M$-generic filter $G\of\Q$ for the class forcing
$\Q=M^{\lt\ORD}$ to add a global well-ordering of the
universe---a forcing notion that is $\kappa$-closed for
every cardinal $\kappa$ and hence adds no sets---and then
letting $S$ consist of the classes definable in
$\<M,{\in},G>$ from parameters; the forcing to add $G$
ensures global choice in $\<M,S,{\in}>$, and is unnecessary
if $\<M,{\in}>$ already has a global choice class function.
It follows that \GBC\ is conservative over \ZFC\ in the
sense that any statement purely about sets that is provable
in \GBC\ is also provable in \ZFC.

The method of forcing works over models of \GBC\ just as it
does over models of \ZFC. In particular, if $\P$ is a
partial order in a model $\calM=\<M,S,{\in}>$ of \GBC\ and
$G\of\P$ is $\calM$-generic, then we may define
$\calM[G]=\<M[G],S[G],{\in}>$, where $M[G]$ and $S[G]$
consist of interpreting in the usual way via $G$ the sets
and classes that are $\P$-names in $\calM$. In the case of
set forcing, where $\P\in M$, the analogues of the basic
\ZFC\ forcing lemmas hold also for \GBC, and in this case
the forcing extension $\calM[G]$ continues to satisfy
$\GBC$. Class forcing, in contrast, can destroy \GBC\ just
as it can destroy \ZFC, such as by making every set
countable, to give one badly behaved example. Nevertheless,
there are large families of nicely behaved class-forcing
notions that necessarily preserve \GBC. This includes {\df
progressively closed} forcing in the sense of
\cite{Reitz2006:Dissertation}, that is, forcing notions
that factor for arbitrarily large cardinals $\kappa$ as
$\P_0*\Ptail$, where $\P_0$ is a set and $\Ptail$ is
$\ltkappa$-closed. As it happens, all of the forcing
notions necessary to carry out the constructions in this
article are progressively closed, and so the corresponding
forcing extensions will all satisfy \GBC.

The main theorem of this section is that every countable
model of G\"odel-Bernays set theory can be extended to a
pointwise definable model.

\begin{theorem}\label{Theorem.NGBC}
Every countable model of G\"odel-Bernays set theory has a
pointwise definable extension, where every set and class is
first-order definable without parameters.
\end{theorem}

Thus, the original countable \GBC\ model $\<M,S,{\in}>$ is
extended to a pointwise-definable \ZFC\ model
$\<M[G],{\in}>$, such that every set in $M[G]$ and every
class in $S$ is definable in $\<M[G],{\in}>$ without
parameters.

In order to prove this theorem, we shall first prove a
special case of it, the case where the original model
$\<M,S,{\in}>$ is {\df principal}, meaning that there is a
class $X\in S$ such that every class in $S$ is first-order
definable in $\<M,{\in},X>$ with set parameters. The
principal models of \GBC\ include many natural instances.
For example, earlier we explained that any model of \ZFC\
can be transformed into a \GBC\ model with the same sets,
by first forcing global choice and then augmenting with the
classes definable from that generic well-ordering; the
resulting \GBC\ model is principal by construction.
Furthermore, the collection of principal models of \GBC\ is
closed under forcing, even proper class forcing (provided
\GBC\ itself is preserved), since one can amalgamate the
original generating class with the newly generic class to
form a new generating class. So we have plenty of principal
models of \GBC. Nevertheless, it is also easy to construct
non-principal models, such as the weak limit of a suitable
iteration of $\omega$ many extensions $M\of M[G_0]\of
M[G_0][G_1]\of\cdots$, whose union satisfies \GBC\ but is
nonprincipal. Additional non-principal models are provided
by the fact that no model $\<M,S,\in>$ of Kelly-Morse set
theory can be principal as a \GBC\ model, since \KM\ proves
that every class $X$ admits a first-order truth class for
$\<M,\in,X>$, since for each $n$ the
$\Undertilde\Sigma_n(X)$ truth class is unique satisfying
the recursive definition of truth and these classes can be
unified into one class in \KM, but no such truth class can
be first-order definable from $X$ over $M$. An instance of
this arises from an inaccessible cardinal $\kappa$ with the
structure $\<V_\kappa,V_{\kappa+1},{\in}>$, which is a
model of \GBC\ and in fact \KM, but is easily seen not to
be principal as a \GBC\ model on cardinality grounds, since
it has $2^\kappa$ many classes but only $\kappa$ many sets
and hence only $\kappa$ many classes definable from any
fixed class. Nevertheless, this model can be made principal
by the forcing to collapse $2^\kappa$ to $\kappa$, which
adds no sets to $V_\kappa$, but which adds a generic
$\kappa$-enumeration of the ground model subsets of
$V_\kappa$.

\begin{theorem}\label{Theorem.PrincipalGB}
Every countable principal \GBC\ model has a class forcing
extension that is pointwise definable, in which every set
and class is first-order definable without parameters.
\end{theorem}

\begin{proof}
Suppose that $\calM=\<M,S,{\in}>$ is a principal countable
model of GBC. To prove the theorem we must show there is a
class forcing extension  $\<M[G],{\in}>$, satisfying \ZFC\,
such that every set in $M[G]$ and every class in $S$ is
definable in $\<M[G],{\in}>$ without parameters. The proof
will follow the proof of theorem
\ref{Theorem.PDForcingExtension}, extending the argument to
handle the issues arising on account of the second-order
part of the models. Since $\calM$ is principal, there is a
class $X\in S$ such that every class in $S$ is first-order
definable in the structure $\<M,{\in},X>$. By forcing if
necessary, we may assume $\calM\satisfies\GCH$.

The first step is to prove the analogue of Simpson's
theorem for this case. Specifically, we claim that there is
an amenable class $U$ such that $\<M,{\in},U>$ satisfies
$\ZFC(U)$ and every set and class in $\<M,S,{\in}>$ is
definable in $\<M,{\in},U>$ without parameters. The proof
is just as in theorem \ref{Theorem.Simpson}, with an extra
step. Enumerate $M=\<a_n\st n<\omega>$ and choose codes
$\tilde a_n\in 2^{\lt\ORD}$ for each $a_n$. Let
$\Q=2^{<\ORD}$ and construct an $\calM$-generic class
$U_0\subset\Q$ as in the proof of
theorem~\ref{Theorem.Simpson} so that every set in $M$ is
definable without parameters in $\<M,{\in},U_0>$. Since the
forcing is $\kappa$-closed for every cardinal $\kappa$, it
adds no sets and
$\calM[U_0]=\<M,S[U_0],{\in}>\satisfies\GBC$, where
$S[U_0]$ is the collection of classes obtained by
interpreting via $U_0$ the $\Q$-names in $\calM$. Let $U$
be the class $X\oplus U_0$, combining the two classes in
some canonical manner as a class of ordinals, and observe
that $\<M,{\in},U>$ satisfies $\ZFC(U)$, because the
forcing $\Q$ is progressively closed -- indeed, it is
$\kappa$-closed for every cardinal $\kappa$. Since $U_0$ is
definable from $U$ in $\<M,{\in},U>$, it follows that every
element of $M$ is definable in $\<M,{\in},U>$ without
parameters. And since $X$ is also definable from $U$ there,
and every class in $S$ is definable from $X$ with set
parameters, it follows that every class in $S$ is also
definable in $\<M,{\in},U>$ without parameters, as desired.

As in theorem \ref{Theorem.PDForcingExtension}, we may now
force to code the digits of $U$ into the \GCH\ pattern at
the cardinals $\delta_\alpha=\aleph_{\omega\cdot\alpha+1}$,
producing a forcing extension $M[G]$ in which $U$ is
definable without parameters. It follows that $M$ is also
definable there, since the sets in $M$ are exactly the sets
that are coded into a block of $U_0$. Next, a further
forcing extension $M[G][H]$ satisfies $V=\HOD$ by
generically coding the \GCH\ pattern at the cardinals
$\gamma_\alpha=\aleph_{\omega\cdot\alpha+5}$ as before.
Thus, every element of $M[G][H]$ is definable from ordinal
parameters, but these are definable without parameters
since every element of $M$, including every ordinal, is
definable from $M$ and $U_0$, which are definable from $U$,
which is definable without parameters in
$\<M[G][H],{\in}>$. It follows that every class in $S$ is
definable without parameters in $\<M[G][H],{\in}>$, since
$X$ is definable from $U$ there, and $U$ is definable
without parameters. So every set and class in
$\<M[G][H],{\in}>$ is definable without parameters, as
desired.
\end{proof}

To complete the proof of theorem \ref{Theorem.NGBC}, it
suffices to show that every countable \GBC\ model $\calM$
can be extended to a principal \GBC\ model. Our initial
attempts to prove this involved meta-class forcing. We
wanted to replicate the situation of the model
$\<V_\kappa,V_{\kappa+1},\in>$ in the case $\kappa$ is an
inaccessible cardinal, since this structure is
non-principal, but is made principle in the forcing
extension $V[G]$, where $G\of\Coll(\kappa,V_{\kappa+1})$,
which collapses $2^\kappa$ to $\kappa$. The generic filter
$G$ can be viewed as a subclass of $\kappa\times V_\kappa$,
whose columns are exactly the ground model subsets of
$V_\kappa$, which are exactly the classes of the original
model. In this way, the old classes are unified into one
new generic class. The forcing in this argument is
meta-class forcing, as opposed to merely class forcing,
since the individual conditions are classes in
$\<V_\kappa,V_{\kappa+1},\in>$.

A similar idea appears to work over any model of
Kelly-Morse set theory, although the details of verifying
the meta-class forcing technology are abundant. The idea is
that given a model $\<M,S,\in>\satisfies\KM$, one considers
the meta-class of conditions consisting of an ordinal
$\alpha$ and a subclass $A\of\alpha\times M$, with $A\in
S$. Stronger conditions enlarge $\alpha$ and specify
additional columns, and so the forcing is exactly analogous
to $\Coll(\ORD^M,S)$ as above. In this way, it is dense to
add any class of $S$ onto a vertical slice, and the generic
class $G\of\ORD\times M$ will unify all the classes $X\in
S$ into one class, making a principal model.

Such a meta-class forcing method, however, does not seem to
work easily over \GBC\ models, and one should not expect to
add a single generic class from which all the classes of
$S$ are uniformly definable, by varying only set
parameters.

Nevertheless, we avoid the difficulties of these
higher-order forcing arguments and complete the proof of
theorem \ref{Theorem.NGBC} by appealing to the following
remarkable observation of Sy Friedman (sketched in an email
correspondence with the first author), which shows that
every countable \GBC\ model does indeed have a principal
extension with the same sets. Kossak and Schmerl prove a
similar fact  about models of \PA\ in \cite[Theorem
6.5.6]{KossakSchmerl2006:TheStructureOfModelsOfPA}, showing
that for any model $\<M,R_0,R_1,\ldots>\satisfies\PA^*$
there is a $G\of M$ such that $\<M,G>\satisfies\PA^*$ and
each $R_i$ is definable in $\<M,G>$. It is the
Kossak/Schmerl proof method that we shall adapt to \GBC\
models here. Viewing the original Kossak/Schmerl theorem as
one about models of $\GBC^{\neg\infty}$, the theory $\GBC$
with the infinity axiom replaced by its negation (and
augmented with the assertion that every set has a
transitive closure), Ali Enayat has described theorem
\ref{Theorem.FriedmanPrincipal} here as one of the positive
instances where a model-theoretic theorem about
$\GBC^{\neg\infty}$ extends to $\GBC$; not all of them do.

\begin{theorem}[Friedman/Kossak/Schmerl]\label{Theorem.FriedmanPrincipal}
Every countable model $\calM$ of \GBC\ can be extended to a
principal model $\calM[Y]\satisfies\GBC$ while adding no
sets, only classes. Indeed, in the extension, every set and
class in $\calM$ is definable in the structure
$\<M,\in,Y>$.
\end{theorem}

\begin{proof} For clarity, let us mention up-front that the
extension $\calM[Y]$ we produce will not necessarily be a
class forcing extension of $\calM$, although it will
satisfy \GBC\ because of the close connections it has with
a sequence of increasingly partially generic extensions.
Specifically, for each $n$ we will be able to realize
$\calM[Y]$ as a $\Sigma_n$-generic extension of $\calM$ by
a partial order $\Q_n$, meaning that all
$\Undertilde\Sigma_n$ definable dense subclasses of $\Q_n$
are met by the filter, and by increasing $n$ this suffices
to capture the whole of \GBC.

Suppose that $\calM=\<M,S,{\in}>$ is a countable model of
\GBC, meaning that it has countably many sets and classes.
Enumerate the classes of ordinals in $S$ as $\<A_n\st
n<\omega>$, so that $A_n\of\ORD^M$ and every class in $S$
is first-order definable from some $A_n$.  We shall add a
class $Y$ of ordinals in such a way that $\<M,{\in},Y>$
satisfies $\ZFC(Y)$ and every $A_n$ is coded into $Y$,
hidden away by coding on increasingly difficult-to-define
subclasses of ordinals. Specifically, we will construct a
sequence of trees $\Q_0 \supset \Q_1 \supset \Q_2
\supset\cdots$, each definable but with increasingly
complex definitions. The class $Y$ will give a branch
through $\bigcap \Q_n$, and $A_n$ will be determined by the
choices that $Y$ makes within $\Q_n$.  Thus $A_n$ will be
recoverable from $Y$ and $\Q_n$, but the increasing
complexity of $\Q_n$ suggests that we should not expect a
uniform definition.

If we succeed, then it follows by induction that every
$A_n$ and hence every class in $S$, is definable in
$\<M,{\in},Y>$, whose definable classes form a principal
\GBC\ model, and so we will have proved the theorem.

The argument will rely on some refinements of the customary
general facts about class forcing, which we shall summarize
here but not prove. First, for any sufficiently
nice\footnote{The question of whether every class forcing
notion has a definable forcing relation (i.e. `is
sufficiently nice') is open.  However, a sufficient
condition which holds of a great many class forcings,
including all those appearing in this paper, is that every
subset $A\subset \P$ is contained in a complete subposet $Q
\subset_C \P$} class forcing notion $\P$, the forcing
relation $p\forces_\P\varphi$ is definable in $M$ from
$\P$; the refinement we need is that for $\varphi$ of
bounded complexity $\Sigma_n$, the forcing relation has
some bounded complexity $\Sigma_{k}(\P)$ in $M$. Second,
whenever $p\forces_\P\varphi$, then $M[G]\satisfies\varphi$
for any $M$-generic filter $G\of\P$ containing $p$; the
refinement we need is that for $\varphi$ of bounded
complexity $\Sigma_n$, there is some $k$ such that
$M[G]\satisfies\varphi$ for any $\Sigma_{k}(\P)$-generic
filter $G\of\P$ containing $p$. Third, whenever
$M[G]\satisfies\varphi$ for some $M$-generic filter
$G\of\P$, then there is a condition $p\in G$ such that
$p\forces_\P\varphi$; the refinement we need is that for
$\varphi$ of bounded complexity $\Sigma_n$, there is some
$k$ such that whenever $M[G]\satisfies\varphi$ for some
$\Sigma_k(\P)$-generic filter $G\of\P$, then
$p\forces_\P\varphi$ for some $p\in G$. These refinements
simply assert that the forcing mechanics are sound with
respect to partially but sufficiently generic filters, and
can be proved by induction by following the usual proofs of
the forcing lemmas and paying attention to the complexity
of the dense sets that arise in the arguments. It is not
relevant for our application to find the strictly optimal
relations between $n$ and $k$, although one could do this.

Now, let us return to the argument and create the generic
coding class $Y$. Let $\Q_0=\Add(\ORD,1)=2^{\lt\ORD}$ be
the class of all binary ordinal length sequences, ordered
by extension. We will construct a descending sequence of
class forcing notions
$$\Q_0 \supset \Q_1 \supset \Q_2 \supset\cdots$$
Each $\Q_n$ will be a \emph{perfect tree}, that is, a
subclass of $2^{\lt\ORD}$ that is \emph{splitting} (any
node can be extended to two incompatible nodes, though not
necessarily at the next level) and contains full limits.
There is a  canonical embedding of $2^{\lt\ORD}$ into any
perfect tree, whose image is exactly the splitting nodes
(nodes with two immediate successors) of the tree. This
embedding is dense, and therefore witnesses the forcing
equivalence of $\Add(\ORD,1)$ and any perfect tree. Thus,
all the $\Q_n$ are individually forcing equivalent to
$\Add(\ORD,1)$, which is $\kappa$-closed for every cardinal
$\kappa$ and therefore adds no new sets, while forcing
$\ZFC$ relative to the generic class.


Given $\Q_n$, we construct $\Q_{n+1}$ in two stages. First,
we will refine $\Q_n$ to a subtree with the property that
any branch will meet all $\Undertilde\Sigma_n$ definable
dense subclasses of $\Q_n$.  By increasing the length of
the stem, or portion of the tree below the first branch
point,  in this stage we will also guarantee a branch
through the intersection.  We then further refine this
subtree to obtain $\Q_{n+1}$, ensuring that any branch
through $\Q_{n+1}$ codes $A_n$.

The construction proceeds as follows.  We fix a sequence of
ordinals $\<a_n \mid n \in \omega>$ cofinal in $\calM$
(these will govern the growth of the stems of the $\Q_n$).
Given $\Q_n$, we fix an enumeration $\<D_\alpha \mid \alpha
\in \ORD>$ in $\calM$ of all $\Undertilde\Sigma_n(\Q_n)$
definable dense subclasses of $\Q_n$.  We now construct an
embedding $f:2^{\lt\ORD} \to \Q_n$.  Let $f(\emptyset)$ be
the least splitting node of $\Q_n$ which has length
exceeding $a_n$ (by `least' we mean the node of shortest
length, and if more than one such node exists we take the
least according to some fixed class well-ordering of
$\calM$ from $S$).  Given $f(p)$ where $p$ has length
$\alpha$, and fixing $e \in \{0,1\}$, we let $f(p\concat
e)$ be the least splitting node $q$ of $\Q_n$ extending
$f(p)\concat e$ such that $q \in D_\alpha$.    For $p$ of
limit length, we take $f(p)$ to be the least splitting node
extending $\Union_{p^\prime \subset p} f(p^\prime)$ and
lying in $D_\alpha$.  Let $\Q_n^*$ denote the subtree of
$\Q_n$ determined by $f$, that is, the tree consisting of
all predecessors of members of $\ran(f)$.  For each
$\alpha$, any branch through $\Q_n^*$ will necessarily meet
$\{f(p)\mid p \text{ has length } \alpha \} \subset
D_\alpha$, and so any branch will meet all
$\Undertilde\Sigma_n$ definable dense subclasses of $\Q_n$.

We now refine $\Q_n^*$ to encode $A_n$ by selecting only
those $q \in \Q_n^*$ which ``choose according to $A_n$ on
the even splitting nodes.''  That is, we select those $q$
such that, for any predecessor $q^\prime$ of $q$, if
$q^\prime$ is a splitting node of $\Q_n^*$ with exactly
$2\cdot \alpha$ many splitting nodes preceding it, then
$q^\prime \concat 1 \in q$ if and only if $\alpha \in A_n$.
Let $\Q_{n+1}$ be the subtree of $\Q_n^*$ determined by
such $q$, that is, the tree consisting of all predecessors
of such $q$.  Given $\Q_n$ together with any branch of
$\Q_{n+1}$ we can easily recover $A_n$ by comparing values
of the branch corresponding to the even splitting nodes of
$\Q_n$.  Conversely, note that $\Q_{n+1}$ was defined from
parameters $\Q_n$ and $A_n$.

This completes the construction.  We now take $Y = \bigcap
\Q_n$.  That $Y$ is nonempty and contains a single branch
follows from the fact the $\Q_n$ are nested and have stems
of increasing length unbounded in $\ORD^\calM$. The key
observation remaining is that the construction ensures that
$\<M,{\in},Y>$ satisfies $\ZFC(Y)$. Any axiom of $\ZFC(Y)$
has some logical complexity $\Sigma_n$ in the forcing
language, and by the refined forcing facts mentioned above
there is some $k_n$ sufficiently large so that the forcing
lemmas concerning $\Sigma_n$ work as expected for all
$\Sigma_{k_n}(\Q_{k_n})$ generic filters of $\Q_{k_n}$.  By
construction $Y$ has the required level of genericity for
$\Q_{k_n}$, and so $\<M,{\in},Y>$ satisfies every axiom of
$\ZFC(Y)$, as desired.  As we have observed, each $A_n$ is
definable from $Y$ and $\Q_n$.  The partial order
$\Q_{n+1}$ is definable from $\Q_n$ and $A_n$, and it
follows inductively that each $A_n$ and each $\Q_n$ are
definable in $\<M,{\in},Y>$.  Thus every class in $S$ is
definable from $Y$, and so we have produced the desired
principal model $\<M,{\in},Y>$ extending $\calM$. In
particular, every set and class of this model is definable
using only the parameter $Y$ over $M$.
\end{proof}

Theorem \ref{Theorem.NGBC} now follows from theorems
\ref{Theorem.PrincipalGB} and
\ref{Theorem.FriedmanPrincipal}, since any countable model
of \GBC\ can first be made principal by theorem
\ref{Theorem.FriedmanPrincipal} and then pointwise
definable by theorem \ref{Theorem.PrincipalGB}. In fact,
the Friedman/Kossak/Schmerl theorem (theorem
\ref{Theorem.FriedmanPrincipal}) can also serve as a
replacement for Simpson's theorem \ref{Theorem.Simpson} in
the proof of theorem \ref{Theorem.PDForcingExtension},
since as we mentioned every set and class of $\calM$ is
definable in the structure $\<M,\in,Y>$, which was the
whole point of theorem \ref{Theorem.Simpson}.

Finally, we remark that the proof we have just given of
theorem \ref{Theorem.NGBC} allows us to preserve any of the
usual large cardinals from the ground model to the
pointwise definable extension. This is because the
application of theorem \ref{Theorem.FriedmanPrincipal} does
not add sets, and so preserves all large cardinal
properties, and the remaining forcing of theorem
\ref{Theorem.PrincipalGB} is the very nice, progressively
closed Easton-support forcing iteration to code the class
$Y$ and then force $V=\HOD$. As we mentioned after theorem
\ref{Theorem.PDForcingExtension}, this can be easily
arranged to accommodate any of the usual large cardinal
notions. Thus, any countable model of \GBC\ with any of the
usual large cardinals has an extension to a
pointwise-definable model of \ZFC, in which every set and
class is definable without parameters, and in which those
large cardinals are preserved.

\nocite{FuchsHamkinsReitz:Set-theoreticGeology,Reitz2006:Dissertation}
\bibliographystyle{alpha}
\bibliography{MathBiblio,HamkinsBiblio,ExtraBib}

\end{document}